\documentclass[12pt]{amsart}
\usepackage[cp1251]{inputenc}
\usepackage[english]{babel}
\usepackage{amsmath,amsthm,amssymb}
\usepackage{hyperref}
\newtheorem{theorem}{Theorem}
\newtheorem{example}{Example}
\newtheorem*{remark}{Remark}
\DeclareMathOperator{\dist}{dist}

\author{Fedor Petrov}
\thanks{The work is supported by RFBR grants 11-01-00677-a, 13-01-00935-a, 13-01-12422-ofi-m and
President of Russia grant MK-6133.2013.1.}
\email{fedyapetrov@gmail.com}
\address{
St. Petersburg   Department   of   V.~A.~Steklov    Institute   of   Mathematics
RAS, St. Petersburg State University, Yaroslavl State Univeristy.}

\title{General removal lemma}

\begin{document}

\maketitle

In this short note we formulate and prove a general result in spirit of
hypergraph removal lemma. We deal with a continuous version, but there
is a general machinery of Elek and Szegedy \cite{ES}, allowing to get
discrete versions. 

Our result generalizes some concrete lemmata of \cite{J1,J2,ZP}. Statements of this type
are used in the theory of continuous graphs, see the recent book \cite{L} and references therein.
Discrete versions have various applications in combinatorics. 

\begin{theorem}\label{main}
1. Let $K$ be a metric compact space, $X$ be a standard continuous measure space (say, $X=[0,1]$ 
with Lebesgue measure). Given positive integer $k$ and a measurable 
(with respect to the Borel sigma-algebra on $K$)
$K$-valued function $f(x_1,x_2,\dots,x_k)$ on $X^k$.
Denote by ${\mathcal N}_k$ the set of all ordered $k$-tuples $I=(i_1,\dots,i_k)$ of mutully distinct
positive integers. For any such $k$-tuple $I$ and any points $y_1,y_2,\dots$ in $X$ denote 
$f_I(y_1,y_2,\dots):=f(y_{i_1},\dots,y_{i_k})$. So, $f$ induces a map $\tilde{f}$ from 
$X^{\mathbb{N}}$ to the Tychonoff's compact set $K^{{\mathcal N}_k}$, which
maps a point $(y_1,y_2,\dots)\in X^{\mathbb{N}}$ 
to the function $I\rightarrow f_I(y_1,\dots,y_n)$ on ${\mathcal N}_k$. 
Let $M$ be a fixed closed 
(hence compact) subset of $K^{{\mathcal N}_k}$ 
and assume that for almost all $y_1,y_2,\dots$ the value $\tilde{f}(y_1,y_2,\dots)$
belongs to $M$. Then there exists a measurable function $g$ on $X^k$, equivalent to $f$, such that 
$\tilde{g}(y_1,y_2,\dots)$ belongs to $M$ for all mutually different $y_1,y_2,\dots$ in $X$. 

2. Assume additionally that $f$ is ``almost symmetric'', i.e. 
\begin{equation}\label{symmetry}
f(x_1,\dots,x_k)=f(x_{\pi_1},\dots,x_{\pi_k})
\end{equation}
for almost all $x_1,\dots,x_k$ in $X$ and any permutation $\pi$ of the set $\{1,2,\dots,k\}$.
Then there exists a measurable symmetric function $g$ on $X^k$, equivalent to $f$, such that 
$\tilde{g}(y_1,y_2,\dots)$ belongs to $M$ for all (not necessary different) $y_1,y_2,\dots$ in $X$. 
\end{theorem}

Before we pass to the proof, let us mention some examples.

\begin{example} Let $k=2$, $K=\{0,1\}$, and for almost all $y_1,y_2,y_3$ we have 
$$
f(y_1,y_2)=f(y_2,y_1),\, 
f(y_1,y_2)\cdot f(y_2,y_3)\cdot f(y_1,y_3)=0.
$$
(This corresponds to some explicit closed set $M$, of course.) So, $f$ defines a graphon which has almost no triangles. Then the claim is that 
we may save almost all edges so that there will be no triangle at all, i.e. continuous version of the
triangle removal lemma. Similarly we get the hypergraph removal and induced hypergraph removal lemmata and so on.
\end{example}

\begin{example}[\cite{ZP}] Let $k=2$, $K=[0,\infty]$ and for almost all $y_1,y_2,y_3$ we have 
$$
f(y_1,y_2)=f(y_2,y_1),\, 
f(y_1,y_2)+f(y_2,y_3)\geq f(y_1,y_3).
$$
This again corresponds to appropriate closed set $M$. In this case we deal with ``almost metric space'', which therefore may be ``corrected''
by changing distances between null set of pairs to a genuine semimetric space. Values $g(x,x)$ may be redefined as 0, if needed. 
Also, a priori infinite distances may occur. But at fact
almost all distances should be finite, and hence on some set $X'$ of full measure all distances are finite. We may identify the complement 
$X\setminus X'$ of this set with one of points $x_0\in X'$, and so get all distances being finite. 
\end{example}

The proof of part 1 (non-symmetric version) 
consists of two ingredients: Lebesgue density theorem and 
Tychonoff's compactness theorem. In part 2 (symmetric version) we need also
the following standard variant of Ramsey theorem.

\begin{theorem}[Ramsey theorem]\label{Ramsey} Given $c<\infty$ colors, positive integers $k_1,\dots,k_\nu$ and positive integers $N_1,\dots,N_\nu$. 
Then there exist positive integers $R_1,\dots,R_\nu$ so that for disjoint finite sets $A_1,\dots,A_\nu$ of 
cardinalities $|A_i|=R_i$, $1\leq i\leq \nu$, the following statement holds:

assume that each array $(B_1,\dots,B_\nu)$, where $B_i\subset A_i$ and $|B_i|=k_i$ is colored in one 
of our $c$ colors. Then there always exist sets $C_i\subset A_i$, $|C_i|=N_i$, so that colors of 
arrays satisfying $B_i\subset C_i$ are all the same.
\end{theorem}

Identify $X$ with $[0,1)$ equipped by the Lebesgue measure $\mu$
and for $x\in X$ denote by $\Delta_m(x)$ the unique semiinterval $[s/m,(s+1)/m)$,
containing $x$ ($s=0,1,\dots,$).

We need the following variant of the density theorem:

\begin{theorem}\label{density}
For almost all $x_1,\dots,x_k$ in $X$ 
for any open set $U\subset K$ containing $f(x_1,\dots,x_k)$ 
$$
\lim_m m^k\cdot \mu\left(f^{-1}(U)\cap \prod_{i=1}^k \Delta_m(x_i) \right)=1.
$$
\end{theorem}

\begin{proof} Consider a countable base of the topology on $K$. It suffices
to take open sets $U$ from the base. For each 
of them this is just the usual Lebesgue density theorem. 
\end{proof}

Denote by $Y\subset X^k$ the set of full measure for which the condition of Theorem \ref{density} holds.

For positive integer $\nu$ define the metric on $K^{\nu}$ by
$$
\dist\left((x_1,\dots,x_\nu),(y_1,\dots,y_{\nu})\right):=\max_{1\leq i\leq \nu} \dist_K(x_i,y_i). 
$$

\begin{proof}[Proof of Theorem \ref{main}] We start with part 1.

At first, we require that $f$ and $g$ coincide on $Y$. This already implies that $g$ is measurable
and equivalent to $f$.

%We call a point $(y_1,\dots,y_n)\in X^n$ \emph{regular}, if $y_1,\dots,y_n$ are mutually different and 
%$(y_{i_1},\dots,y_{i_k})\in Y$ for all mutually distinct indecies $i_1,\dots,i_k\in\{1,2,\dots,n\}$. 
%Then almost all points in $X^n$ are regular. 

%Note that for the regular point $(y_1,\dots,y_n)\in X^n$ we have $\tilde{f}(y_1,\dots,y_n)=\tilde{g}(y_1,\dots,y_n)\in M.$
%Indeed, assume the contrary: $\tilde{f}(y_1,\dots,y_n)\notin M$. Then, since $M$ is closed, there exist open sets
%$U_I$ for $I\in [n,k]$ containing $f_I(y_1,\dots,y_n)$ such that $\prod_I U_I\cap M=\emptyset$. Take large integer $m$ and choose random points
%$y_i'\in \Delta_m(y_i)$. Then by Theorem \ref{density} with probability tending to 1 (as $m$ grows) for all $I\in [n,k]$ 
%we have $f_I(y_1',\dots,y_n')\in U_I$. But it means that with positive (and even close to 1) probability $\tilde{f}(y_1',\dots,y_n')\notin M$,
%a contradiction.

Now we have to define values of $g$ on $X^k\setminus Y$ so that $g$ satisfies the condition of Theorem \ref{main}.
Denote by $\Phi$ the set of $K$-valued functions $g$ on $X^k$ such that
$g=f$ on $Y$. This set $\Phi$ is a closed subset of a Tychonoff's compact space $K^{X^k}$, which may be naturally
identified with $K^{X^k\setminus Y}$.  

The closed set $M$ is an intersection of closed cylindrical sets, say, $M=\cap_{\alpha} M_{\alpha}$.
For fixed $\alpha$ and fixed mutually different points $y_1,y_2,\dots$ in $X$ the condition 
\begin{equation}\label{alpha}
\tilde{g}(y_1,y_2,\dots)\in M_{\alpha}
\end{equation} 
defines a closed subset of $\Phi$. We have to prove that all those
closed subsets of $\Phi$ share a common point. So, it suffices to prove that any finite collection 
of such subsets share a common point.
Fix such a collection. It deals only with a finite number of $k$-tuples in $X^k$. 
Denote by $A=\{x_1,\dots,x_n\}$ the finite set of all points in those $k$-tuples. 
Then conditions (\ref{alpha}) hold simultaneously iff 
\begin{equation}\label{mshtrix}
\tilde{g}(x_1,\dots,x_n,\dots)\in M',
\end{equation}
where $M'$ denotes the closed cylindrical set 
determined by $k$-tuples of different indecies not exceeding $n$.
Let's write $\tilde{g}(x_1,\dots,x_n)$ for the LHS of (\ref{mshtrix}), 
since further arguments of $\tilde{g}$ are of no importance. 

We have to define $g$ on all $k$-tuples $(z_1,\dots,z_k)\in A^k$ 
so that 

(i) $\tilde{g}(x_1,\dots,x_n)\in M'$; and 

(ii) $g$ coincides with $f$ on $Y\cap A^k$.  

Fix arbitrary $\varepsilon>0$. Assume that we succeed to define 
$g$ so that 

(i$-\varepsilon$) $\tilde{g}(x_1,\dots,x_n)$ is $\varepsilon$-close to $M'$; and

(ii$-\varepsilon$) $g(y_1,\dots,y_k)$ and $f(y_1,\dots,y_k)$ 
are $\varepsilon$-close in $K$ provided that 
$(y_1,\dots,y_k)\in Y\cap A^k$. 

Then let $\varepsilon$
tend to 0 and choose a convergent subsequence of values of $g$ on all $k$-tuples in $A$.
Clearly this limit function satisfies (i) and (ii), as desired. 

For finding $g$ satisfying (i$-\varepsilon$) and (ii$-\varepsilon$) we take large $m$
and replace any point $z\in A$ to a random point $z'\in \Delta_m(z)$. Then define
$$
g(z_1,\dots,z_k)=f(z_1',\dots,z_k')
$$
for any points $z_1,\dots,z_k$ in $A$. 

Then (i) (therefore (i$-\varepsilon$) aswell) holds with probability 1. Due to Theorem \ref{density},
condition (ii$-\varepsilon$) holds with probability arbitrarily close to 1 
provided that $m$ is large enough. So, with positive probability such $g$ works.

Now we pass to proving part 2 of the theorem. 
At first, we change $M$ so that any function $g$ satisfying
$\tilde{g}(y_1,y_2,\dots)\in M$ must be symmetric. For this we intersect $M$ 
with sets defined by $g(y_1,\dots,y_k)=g(y_{\pi_1},y_{\pi_2},\dots,y_{\pi_k})$.
Of course, $\tilde{f}(y_1,y_2,\dots)\in M$ still holds for almost all $y_1,y_2,\dots\in X$.

We follow the lines of proving part 1. In particular, we find a finite subset $A=\{x_1,\dots,x_n\}$.
The difference is that now we need to check not 
just $\tilde{g}(x_1,\dots,x_n)\in M'$ for appropriate $M'$, but 

(i') $\tilde{g}(y_1,\dots,y_n)\in M'$ for all $y_1,\dots,y_n$ in $A$ (not necessary different).

For this on last step 
instead of replacing each point $z\in A$ to a random point $z'\in \Delta_m(z)$ we fix large number
$R$ to be specified later ($R$ depends only on $K$, $\varepsilon$ and $n=|A|$, and does not depend on $m$) 
and for each $z\in A$ choose $R$ independent random points in $\Delta_m(z)$. They 
form a random set $\Omega(z)$ (of course, sets $\Omega(z),z\in A$, are disjoint sets of cardinality
$R$ with probability 1). 
For arbitrary points $z_1,\dots,z_k$ in $A$ we require
$g(z_1,\dots,z_k)=f(z_1',\dots,z_k')$ for some $z_i'\in \Omega(z_i)$.
Then for large enough $m$ condition (ii$-\varepsilon$) holds with probability almost 1 for all possible
choices $z_i'\in \Omega(z_i)$.

We partition $K$ onto a finite number $c$ of parts so that each part has diameter less then $\varepsilon$.
Let's think that those parts correspond to $c$ different colors.
 
For any set $S=\{y_1,\dots,y_k\}\subset \cup_{z\in A} \Omega(z)$, define its color as a part containing 
$f(y_1,\dots,y_k)$ (with probability 1 this is correct, i.e. does not depend on the order of elements in $S$).  
Define a type of a set $S$ as a function $z\rightarrow |I\cap \Omega(z)|$ on $A$. The number of possible
types depends only on $|A|$ and $k$. Applying Theorem \ref{Ramsey} repeatedly 
for large enough $R$ we may find subsets
$\Xi(z)\subset \Omega(z)$, $|\Xi_i|=n$, so that all subsets $S\subset \cup_{z\in A} \Xi(z)$  of the same
type have the same color. Now we are ready to define $g(z_1,\dots,z_k)$ for all (not necessary distinct)
points $z_1,\dots,z_k$ in $A$. 
%For $z\in A$ denote $b(z)=|\{i: z_i=z\}|$. Then sets of type $b$ 
%in $\cup \Xi(z)$ have the same color. 
We require that $g(z_1,\dots,z_k):=f(z_1',\dots,z_k')$, where $z_i'\in \Xi(z_i)$ are mutually different points,
and $g$ is symmetric on $A^k$. Of course, both conditions may be satisfied. 
Note that for other possible choices of $z_i'$ the value of $g$ moves by a distance at most $\varepsilon$. 

We have to check (i'). Choose mutually different points $y_i'\in \Xi(y_i)$, $i=1,2,\dots,n$.
With probability 1 we have $\tilde{f}(y_1',\dots,y_n')\in M'$. Our construction of $g$
guarantees that $\tilde{g}(y_1,\dots,y_n)$ and $\tilde{f}(y_1',\dots,y_n')$ are $\varepsilon$-close in
$K^{n(n-1)\dots (n-k+1)}$, as desired.

\end{proof}

\begin{remark} It is easy to see that requirement of $f$ being symmetric is essential in a part 2.
Say, if $K=\{0,1\}$, $k=2$ and condition to $f$ is $|f(x,y)-f(y,x)|=1$ for \textit{almost all} $x,y\in X$
(there are many such functions), this can not be satisifed for \textit{all} $x,y$. The reason is that  
Ramsey type theorem does not hold for oriented graphs.

Also, we can not replace subset $I$ to multisets even in the symmetric case. 
Say, if $k=2$, $K=\{0,1\}$ and $|f(x,y)-f(x,x)|=1$ for almost all $x,y$ (which holds
for $f(x,y)=\chi_{x\ne y}$), this can not be made true for all $x,y$.
\end{remark}

I would like to thank Pavel Zatitskii for the fruitful discussions,
and especially L\'aszl\'o Lov\'asz and Svante Janson 
for the motivating questions and comments, which led to a sequence of 
generalizations of the main statement and improvements of the exposition.

\end{document}